\documentclass[12pt,reno]{amsart}  
\usepackage{amssymb}
\usepackage{amsmath}
\usepackage{enumitem}
\usepackage{mathrsfs}
\usepackage[english]{babel}
\usepackage[all]{xy}
\usepackage{hyperref}
\usepackage{marginnote}
\usepackage{comment}

\usepackage{stmaryrd}

\usepackage{fullpage}

\RequirePackage{mathrsfs} 

\newtheorem{theorem}{Theorem}[section]
\newtheorem{lemma}[theorem]{Lemma}
\newtheorem{proposition}[theorem]{Proposition}

\newtheorem{conjecture}[theorem]{Conjecture}

\theoremstyle{definition}
\newtheorem*{ack}{Acknowledgements}
\newtheorem*{con}{Conventions}
\newtheorem{remark}[theorem]{Remark}

\newtheorem{definition}[theorem]{Definition}

\numberwithin{equation}{section} \numberwithin{figure}{section}

\DeclareMathOperator{\Spec}{Spec}

\DeclareMathOperator{\an}{an}

\newcommand\QQ{\mathbb{Q}}

\newcommand\CC{\mathbb{C}}

\usepackage{color}

\definecolor{orange}{rgb}{1,0.5,0}

\title[Rational points and ramified covers of products of two elliptic curves]{Rational points and ramified covers of products of two elliptic curves}

\author{Ariyan Javanpeykar}
\address{Ariyan Javanpeykar \\
Institut f\"{u}r Mathematik\\
Johannes Gutenberg-Universit\"{a}t Mainz\\
Staudingerweg 9, 55099 Mainz\\
Germany.}
\email{peykar@uni-mainz.de}

\subjclass[2010]
{14G99 
(11G35,  
14G05,  
32Q45)} 

\keywords{Integral points,  special varieties,  abelian varieties, Hilbert property, Hilbert's irreducibility theorem}

\begin{document}

 \begin{abstract}  
 Corvaja and Zannier conjectured that an abelian variety  over a number field  satisfies a   version of the Hilbert property. We  investigate their conjecture for products of  elliptic curves using Kawamata's structure result for ramified covers of abelian varieties, and Faltings's finiteness theorem for rational points on higher genus curves. 
\end{abstract}

\maketitle

\thispagestyle{empty}

  \section{Introduction}
   Corvaja and Zannier conjectured that an abelian variety  over a number field  satisfies a modified version of the Hilbert property. We  investigate their conjecture for products of  elliptic curves using Kawamata's structure result for ramified covers of abelian varieties, and Faltings's finiteness theorem for rational points on higher genus curves.

  Recall that a normal integral variety $X$ over a field $k$ satisfies the Hilbert property over $k$ (as defined in \cite[\S 4]{Serre}) if, for every positive integer $n$ and every collection of finite surjective morphisms $\pi_i:Y_i\to X$, $1\leq i\leq n$, with $Y_i$ geometrically integral over $k$ and $\deg \pi_i \geq 2$, the set $X(k)\setminus \cup_{i=1}^n \pi_i(Y_i(k))$ is dense in $X$. In particular, if $X$ satisfies the Hilbert property over $k$, then $X(k)$ is dense. The Hilbert property   is closely related to the inverse Galois problem for $\QQ$; see \cite[\S4]{Serre}. In this paper we study a \emph{modified} version of the Hilbert property, motivated by conjectures of Campana and Corvaja-Zannier on rational points for varieties over number fields.

By Hilbert's Irreducibility Theorem \cite[Theorem~3.4.1]{Serre}, a rational variety   over a number field satisfies the Hilbert property. On the other hand,  an abelian variety over a number field does not satisfy the Hilbert property.  Nonetheless,  despite the failure of the Hilbert property for abelian varieties, Lang's conjecture on rational points of pseudo-hyperbolic varieties (see \cite{Lang1}) predicts that abelian varieties should satisfy  a \emph{modified} version of the Hilbert property.  

The aim of this paper is to investigate such modified   Hilbert properties for products of elliptic curves. We start with the ``very-weak-Hilbert property''.  This notion is obtained by restricting oneself, in the definition of the Hilbert property (see \cite[\S 3]{Serre}), to ramified covers and to  ``single'' covers.
  
  \begin{definition}
  Let $k$ be a field. A   normal projective geometrically connected variety  $X$ over   $k$ satisfies the \emph{very-weak-Hilbert property   over $k$} if, for every   finite surjective ramified morphism $\pi:Y\to X$ with $Y$ geometrically integral and normal, the set
  $
  X(k) \setminus  \pi(Y(k))
  $ is dense in $X$.
  \end{definition}
  
 If $k$ is a finitely generated field of characteristic zero and $A$ is an abelian variety over $k$, then the Mordell-Weil and Lang-N\'eron theorem imply that $A(k)$ is a finitely generated abelian group; see \cite[Corollary~7.2]{ConradTrace}.  
We prove that the product $\prod_{i=1}^n E_i$ of elliptic curves $E_1,\ldots, E_n$ over a number field  satisfies the very-weak-Hilbert property under the (necessary) assumption that the rank of each   $E_i(k)$ is positive.

\begin{theorem}\label{thm2} Let $k$ be a finitely generated field of characteristic zero, and let $E_1, \ldots, E_n$ be elliptic curves over $k$ with positive rank over $k$. Then $\prod_{i=1}^n E_i$  satisfies the very-weak-Hilbert property over $k$.
\end{theorem}

  We were first led to investigate the very-weak-Hilbert property for abelian varieties by the work of Corvaja-Zannier on the Hilbert property for the Fermat K3 surface \cite{CZHP}, the work of Coccia on the ``affine'' Hilbert property \cite{Coccia},   Demeio's extensions of Corvaja-Zannier's work \cite{Demeio1, Demeio2},   Streeter's verification of the Hilbert property for certain del Pezzo surfaces \cite{Streeter}, and Zannier's seminal work on Hilbert's irreducibility theorem for powers of elliptic curves \cite{ZannierDuke}.
  
 Let us recall that  in \cite{CZHP} Corvaja and Zannier introduced the following  modified version of the Hilbert property.
  
  \begin{definition}[Corvaja-Zannier] Let $k$ be a field. A   normal projective geometrically connected variety  $X$ over   $k$ satisfies the \emph{weak-Hilbert property   over $k$} if, for every integer $n\geq 1$ and finite surjective ramified morphisms $\pi_i:Y_i\to X$ with $Y_i$ a geometrically integral   normal variety over $k$ ($i=1,\ldots, n$), the set
  \[
  X(k) \setminus \cup_{i=1}^n \pi_i(Y_i(k))
  \] is dense in $X$.
  \end{definition}
  
   Our second result is that the product of two elliptic curves with positive rank satisfies the weak-Hilbert property. This modest contribution requires   the input of Kawamata's extension of Ueno's fibration theorem for closed subvarieties of abelian varieties to the case of ramified covers of products of two elliptic curves (see Theorem \ref{thm:kawamata_fibn}), and uses Faltings's finiteness theorem for higher genus curves in several ways.
  
  \begin{theorem}\label{thm33} Let $k$ be a finitely generated field of characteristic zero,
and  let $E_1$ and $E_2$ be elliptic curves over $k$. If $E_1(k)$ and $E_2(k)$ have positive rank, then $E_1\times E_2$ has the weak-Hilbert property over $k$.
  \end{theorem}

Note that, if $X$ satisfies the weak-Hilbert property over $k$, then $X$ satisfies the very-weak-Hilbert property over $k$. However,   the very-weak-Hilbert property defined above differs \emph{a priori} from Corvaja-Zannier's definition.  Nonetheless, it seems reasonable to suspect that these notions are equivalent.
  
  Clearly,  
  a normal projective geometrically connected variety $X$ over a field $k$ with the Hilbert property (as defined in \cite[\S 3]{Serre}) satisfies the weak-Hilbert property over $k$. Thus, in particular, by Hilbert's irreducibility theorem, any rational variety over a number field $k$  satisfies the weak-Hilbert property over $k$ and, in particular, the very-weak-Hilbert property over $k$.  
  
By  \cite[Theorem~1.6]{CZHP}, if $X$ is a smooth projective geometrically connected variety  over a number field $k$ with the weak-Hilbert property, then $X$    satisfies the Hilbert property over $k$ if and only if it is geometrically simply-connected (i.e., $\pi_1^{et}(X_{\overline{k}}) = {1}$).  Indeed, by \emph{loc. cit.}, a smooth projective geometrically connected variety $X$ over a number field $k$ with the Hilbert property is geometrically simply-connected.  In particular, since abelian varieties over number fields are not geometrically-simply connected, they do not have the Hilbert property.

  Corvaja and Zannier conjectured that a smooth projective geometrically connected variety $X$ over a  number field $k$ for which the set $X(k)$ is dense satisfies the weak-Hilbert property over a finite field extension of $k$.   We    state  Corvaja-Zannier's conjecture in the slightly more general context of varieties over finitely generated fields of characteristic zero, and also include the implied (currently not known) equivalence between the weak-Hilbert property and the very-weak-Hilbert property (up to a finite field extension).
  
  \begin{conjecture}[Corvaja-Zannier]\label{conj} Let $X$ be a smooth projective geometrically connected variety over a finitely generated field $k$ of characteristic zero. Then the following statements are equivalent.
  \begin{enumerate}
\item There is a finite extension $M/k$ such that $X_M$ satisfies    the  weak-Hilbert property over $M$.
  \item There is a finite exension $N/k$ such that $X_N$ satisfies the very-weak-Hilbert property over $N$.
  \item    There is a finite extension $L/k$ such that $X(L$) is Zariski-dense in $X$;
  \end{enumerate}
  \end{conjecture}

    Campana's conjectures on ``special'' varieties provide another perspective on Conjecture \ref{conj}. Indeed, Campana conjectured that $(3)$ (and thus also $(1)$ and $(2)$) should be equivalent to $X_{\overline{k}}$ being special; see \cite[Conjecture~9.20]{CampanaOr0} (and also \cite{CampanaBook}). Examples of special varieties are abelian  varieties, K3 surfaces, and rationally connected smooth projective varieties.  Such varieties are thus expected (guided by the above conjectures) to satisfy the weak-Hilbert property over some finite extension of the finitely generated base field $k$ of characteristic zero. Proving that such varieties satisfy the weak-Hilbert property seems very difficult, as it is currently not even known whether all K3 surfaces or Fano varieties have a potentially dense set of rational points. We will comment a bit more on Campana's conjectures below.

 In \cite{FreyJarden} Frey-Jarden proved that   an abelian variety $A$ over a finitely generated field $k$ of characteristic zero admits a finite extension $L/k$ such that $A(L)$ is  Zariski-dense in $X$ (see also  \cite[\S3]{HassettTschinkel} or \cite[\S3]{JAut}).
Thus,  Corvaja-Zannier's conjecture (Conjecture \ref{conj}) predicts that an abelian variety $A$ over a finitely generated field $k$ of characteristic zero satisfies the weak-Hilbert property over some finite field extension of $k$.   Theorems \ref{thm2}  and \ref{thm33}    provide evidence for Corvaja-Zannier's conjecture.

The fact that  an elliptic curve of positive rank over $k$ satisfies the weak-Hilbert-property is already known and is,  as noted in \cite{CZHP}, a consequence of   Faltings's theorem (\emph{quondam} Mordell's conjecture) \cite{Faltings2, FaltingsComplements}.    

Our  results (Theorem \ref{thm2} and Theorem \ref{thm33}) generalize earlier work of Zannier in which evidence for     Conjecture \ref{conj} was provided for abelian varieties $A$ which are isogenous to $E^n$ with $E$ a non-CM elliptic curve \cite{ZannierDuke, ZannierRendi}.  Note that Zannier's arguments are very different from ours and   rely on Hilbertian properties of cyclotomic fields (see \cite{DvornZann, ZannierPisot}).
Theorem \ref{thm2} also provides a non-linear analogue of Corvaja's theorem for linear algebraic groups \cite{CorvajaHilb}.
  
Since elliptic curves of positive rank over a number field satisfy the weak-Hilbert property, the most natural approach to proving that the product of elliptic curves satisfies the weak-Hilbert property would be to  show  that the product of two varieties satisfying the weak-Hilbert property over $k$ satisfies the weak-Hilbert property. This product property seems however  difficult to establish. Instead, to prove Theorem \ref{thm2}, we verify a ``weaker'' expectation.

\begin{theorem}\label{thm3} Let $k$ be a field and $X_1,\ldots, X_n$ be  integral normal projective varieties over $k$.
Assume that, for every $i=1,\ldots, n$, the variety  $X_i$ satisfies   the weak-Hilbert property over $k$. Then $X_1\times \ldots \times X_n$ satisfies the very-weak-Hilbert property over $k$.
\end{theorem}

  Our approach to  Theorem \ref{thm3} is inspired greatly by   the arguments of Bary-Soroker--Fehm--Petersen \cite{BarySoroker}. Indeed, in \emph{loc. cit.} it is shown that, if $X$ and $Y$ satisfy the Hilbert property over $k$, then $X\times Y$ satisfies the Hilbert property over $k$.  Their result answers an old question of Serre in the positive (see the Problem stated in \cite[\S3.1]{Serre}). We mention that Bary-Soroker--Fehm--Petersen's  product theorem for varieties with the Hilbert property can also be deduced from \cite[Lemma~8.12]{HarpazWittenberg} (which builds on Wittenberg's thesis \cite[Lemma~3.12]{Wittenberg}).

The most general criterion we prove for verifying the very-weak-Hilbert property for a variety is Theorem \ref{prop:thm}. It is precisely this result which was inspired by Bary-Soroker--Fehm--Petersen's work \cite{BarySoroker}.

Let us briefly mention that   Theorem \ref{thm3} has further consequences.    For example, if  $E$ is an elliptic curve over a finitely generated field  $k$ of characteristic zero with $E(k)$ of positive rank,  then the variety $E^n \times \mathbb{P}^m_k$ satisfies the very-weak-Hilbert property over $k$.  Moreover, if $X$ is the K3 surface defined by $x^4+y^4=z^4+w^4$ in $\mathbb{P}^3_{k}$, then $E^n\times X$ also satisfies the very-weak-Hilbert property over $k$, as Corvaja-Zannier proved that $X$ satisfies the Hilbert property over $k$ (see \cite[Theorem~1.4]{CZHP}).  
 
\subsection{Campana's conjectures}   Campana's aforementioned notion of special variety forms an important guiding principle in our study of varieties with the weak-Hilbert property. 
 In fact, Campana's conjectures reach much further and also predict a precise interplay between density of rational points and    dense entire curves (much like Lang's conjectures \cite{Lang1}); this is also hinted at by Corvaja-Zannier (see \cite[\S 2.4]{CZHP}). 
 
 To explain this, let us say that a variety $X$ over $\mathbb{C}$ satisfies the \emph{Brody-Hilbert property} if, for every 
  integer $n\geq 1$ and finite surjective ramified morphisms $\pi_i:Y_i\to X$ with $Y_i$ integral and normal ($i=1,\ldots, n$), there is a holomorphic map $\mathbb{C}\to X^{\an}$ with Zariski-dense image  which does not lift to  any of the covers $\pi_i^{\an}:Y_i^{\an}\to X^{\an}$. A special smooth projective connected variety over $\mathbb{C}$ is   conjectured to satisfy the Brody-Hilbert property; see \cite{CampanaBook}.  In this direction it was shown recently by Campana-Winkelmann that a rationally connected variety over $\CC$ satisfies the Brody-Hilbert property; see \cite{CampanaWinkelmann}. 
  We  also mention   that   an abelian variety $A$ over $\mathbb{C}$ satisfies the Brody-Hilbert property. To see this, given  a ramified cover $X\to A$   with $X$ a  normal integral variety over $\CC$, note that a dense entire curve $\CC\to A^{\an}$  which is transversal to the branch locus of $X\to A$ does not lift to $X^{\an}$.    On the other hand, it is not known whether every K3 surface satisfies the Brody-Hilbert property, as we do not know whether such surfaces admit a dense entire curve.

This being said, our motivation for writing this short note is to   call some attention to the beautiful string of new ideas surrounding the weak-Hilbert property, potential density of rational points on varieties over number fields, the existence of dense entire curves, and Campana's special varieties. In fact,  we were naturally led to investigating these problems by our work on Lang's conjectures  \cite{Lang1} (see \cite{vBJK, JBook, JKa, JLevin, JLitt, JXie}.)

 \begin{ack}
 We are grateful to Fr\'ed\'eric Campana for many useful and inspiring discussions. We thank David Holmes and Siddharth Mathur for helpful discussions about unramified morphisms. We thank Raymond van Bommel and Olivier Wittenberg for   their help in finding a simple proof of Lemma \ref{lem:dom0}. We gratefully acknowledge support from the IHES. We thank the referee for several useful comments.
 \end{ack}
 
 \begin{con}
 If $k$ is a field, then a variety over $k$ is a finite type separated scheme over $k$. If $X$ and $Y$ are varieties over $k$, then we let $X\times Y$ denote the fiber product $X\times_{\Spec k} Y$. A field $k$ is said to be \emph{finitely generated} if it is finitely generated over its prime field.   We follow the stacks project and say that  a  morphism of schemes is \emph{unramified} if it  is unramified at every point of $X$; see  \cite[Tag~02G3]{stacks-project}.
 A morphism of schemes $f:X\to Y$ is \emph{ramified} if it is not unramified. If $X\to S$ is a morphism of schemes and $s\in S$, then $X_s$ denotes the scheme-theoretic fibre  of $X$ over $s$.
 \end{con}

  \section{The very-weak-Hilbert property}  
Throughout this section, let $k$ be a field. Moreover, 
let $f:X\to S$ be a   morphism of smooth projective    integral varieties over $k$. Furthermore,   
let $\pi:Y\to X$ be a finite surjective ramified morphism  and let \[\xymatrix{   & & Y \ar[ddll] \ar[d]^{\pi} \\  & & X \ar[d] \\ T \ar[rr]^{\psi} & & S }\] be the Stein factorization of the composed morphism $Y\to X\to S$ with   $T$ projective normal integral over $k$ and   the geometric fibers of $Y\to T$   connected (see    \cite[\S III.11]{Har}).

  \begin{proposition}  \label{prop1} Let $U\subset S$ be a dense open subset. Assume that  $S$ satisfies the very-weak-Hilbert property over $k$ and that, for every $s$ in $U(k)\setminus \psi(T(k))$, the set $X_s(k)$ is dense in $X_s$.  If the morphism $\psi:T\to S$ is ramified, then     $X(k)\setminus \pi(Y(k))$ is dense. 
  \end{proposition}
  \begin{proof}   Since $S$ satisfies the very-weak-Hilbert property over $k$ and $T\to S$ is a ramified finite surjective morphism with $T$ a normal integral variety over $k$, the set     $S(k)\setminus \psi(T(k))$ is dense in $S$. In particular, the set $U(k)\setminus \psi(T(k))$ is dense in $S$. Now, note that the set 
    \[
  \bigcup_{s\in U(k)\setminus \psi(T(k))} X_s(k).
  \] is dense in $X$. Indeed, since $X_s(k)$ is dense in $X_s$, the  closure  of $  \bigcup_{s\in U(k)\setminus \psi(T(k))} X_s(k)$ in $X$ contains the dense set   $\cup_{s\in U(k)\setminus \psi(T(k))} X_s$.  Now, note that    $X(k)\setminus \pi(Y(k))$   contains the (dense) set
  \[
  \bigcup_{s\in U(k)\setminus \psi(T(k))} X_s(k). 
  \]  This concludes the proof.
  \end{proof}

  \begin{lemma}\label{lemmatje} Assume that the branch locus $D$ of $\pi:Y\to X$ dominates $S$ (i.e., $f(D) = S$). Then, for every point $s$ in $S$, the morphism $Y_s\to X_s$ is finite surjective ramified.
  \end{lemma}
  \begin{proof} 
  A morphism of varieties $V\to W$ over $k$ is unramified if and only if, for every $w $ in $W$, the morphism $V_w\to \Spec k(w)$ is unramified (i.e., \'etale); see \cite[Tag~00UV]{stacks-project}. Now, let $s$ be a point of $S$. To show that the finite surjective morphism $Y_s\to X_s$ is ramified, let $d\in D$ be a point lying over $s$.  Then, by the definition of the branch locus, $Y_d\to \Spec k(d)$ is ramified. Note that $Y_d = Y_s\times_{X_s} d$ as schemes over $d=\Spec k(d)$. As the fibre of $Y_s\to X_s$ over $d$ is ramified, it follows that $Y_s\to X_s$ is ramified.  
  \end{proof}

  \begin{theorem}\label{prop:thm}  
 Let $U\subset S$ be a dense open subscheme of $S$. Assume that the following statements hold.
\begin{enumerate}
\item The variety  $S$ satisfies the very-weak-Hilbert property over $k$.
\item For every $s$ in $U(k)$, the projective variety $X_s$ is normal  integral and satisfies the weak-Hilbert property over $k$.   
\item The branch locus $D$  of $\pi:Y\to X$ dominates $S$, i.e., $f(D) = S$. 
\end{enumerate}  Then     $X(k)\setminus \pi(Y(k))$ is dense in $X$.
  \end{theorem}
  \begin{proof}
  If $\psi:T\to S$ is ramified, then it follows from Proposition \ref{prop1} that $X(k)\setminus \pi(Y(k))$ is dense in $X$. (We do not need here the assumption that $f(D) = S$.) Thus, to prove the theorem, we may and do assume that $\psi:T\to S$ is unramified. Since $S$ is smooth and $\psi:T\to S$ is a finite surjective unramified morphism, it follows that $T$ is smooth, so that $\psi:T\to S$ is in fact flat, hence \'etale. 
  
  Note that we have a commutative diagram of morphisms
  \[
  \xymatrix{& & Y_T \ar[d]_{\pi_T} \ar[rr] & & Y\ar[d]^{\pi} & &    \\  D_T \ar[rrd]_{\textrm{surjective}} & & X_T \ar[d]_{f_T} \ar[rr]^{\textrm{finite \'etale}} & & X  \ar[d]^{f} & & D \ar[dll]^{\textrm{surjective}} \\ & &  T \ar[rr]_{\psi}^{\textrm{finite \'etale}} & & S & & }
  \]
As the branch locus $D$ of $\pi$ dominates $S$, it follows that the branch locus $D_T$ of $\pi_T:Y_T\to X_T$ dominates $T$. This implies that, for all $t$ in $T$, the morphism $Y_t\to X_t$ is ramified (Lemma \ref{lemmatje}). We now use this observation.

For $s\in U(k)$, consider the finite surjective morphism $Y_s\to X_s$. Let $\{t_1,\ldots, t_r\} = \psi^{-1}\{s\}$. Then $Y_s = Y_{t_1}\sqcup \ldots \sqcup Y_{t_r}$ and, as explained above, every  induced finite surjective morphism $\pi_{s,j}:Y_{t_j}\to X_s$ is ramified. Since every $Y_{t_i}$ is integral and normal and $X_s$ satisfies the weak-Hilbert property over $k$, it follows that 
\[
X_s(k)\setminus \cup_{j=1}^r \pi_{s,j}(Y_{t_j}(k)) = X_s(k) \setminus \pi_s(Y_s(k))
\] is dense in $X_s$. Since, for every $s$ in $U(k)$, the latter set is dense in $X_s$, we conclude that $X(k)\setminus \pi(Y(k))$ is dense in $X$, as required.
  \end{proof}

\section{Products of varieties}
To study products of varieties $X_1,\ldots, X_n$, we will exploit the many projections such a product is equipped with. 
  \begin{definition}\label{defn}
  Let $X_1,\ldots, X_n$ be varieties over $k$ and let $X:=X_1\times \times \ldots \times X_n$. Define $\widetilde{X_i}$ to be the product of $X_1, \ldots, X_{i-1}, X_{i+1},\ldots, X_n$. We let $p_i:X\to \widetilde{X_i}$ be the natural projection.
  \end{definition}
  
  We include a brief proof of the following simple observation.
  
 \begin{lemma}\label{lem:dom0}
  Let $X_1,\ldots, X_n$ be   smooth projective geometrically integral varieties over $k$, and let $D\subset \prod_{i=1}^n X_i$ be a non-empty closed subscheme   of codimension one. Then, there is an integer $j\in \{1,\ldots, n\}$ such that   $p_j(D) = \widetilde{X_j}$.
 \end{lemma}
  \begin{proof} 
  We argue by induction on $n$.  We may and do assume that $D$ is integral. Write $X =\prod_{i=1}^n X_i$. Note that 
  $$D\subseteq X_1 \times p_1(D) \subseteq X.$$
  If $X_1\times p_1(D) = X$, then $p_1(D) = \widetilde{X_1}$, as required. Thus, we may assume that $X_1\times p_1(D)\neq X$. Then, as $D$ is of codimension one, it follows that $D = X_1 \times p_1(D)$. In this case, as $p_1(D)$ is integral and of codimension one in $\widetilde{X_1}$,  after relabeling if necessary,  it follows from the induction hypothesis that $p_1(D)$ surjects onto $X_3\times \ldots\times X_n$. This implies that $D = X_1\times p_1(D)$ surjects onto $\widetilde{X_2} = X_1\times X_3\times \ldots X_n$, as required. 
  \end{proof}
  
  \begin{lemma}\label{lem:dom}   Let $X_1,\ldots, X_n$ be   smooth projective geometrically integral varieties over $k$, and let $\pi:Y\to X_1\times \ldots \times X_n$ be a finite surjective ramified morphism with $Y$ an integral normal projective variety.  Let $D$ be the branch locus of $\pi$. Then  there is an integer $j\in \{1,\ldots, n\}$ such that $p_j(D) = \widetilde{X_j}$.  
  \end{lemma}
  \begin{proof}     Note that $D$ is non-empty, as $\pi$ is ramified.
 Then,  by Zariski-Nagata purity \cite[Theorem~X.3.1]{SGA1}, the branch locus $D$ is a closed subscheme pure of codimension one, so that the lemma follows from Lemma \ref{lem:dom0}.  
  \end{proof}

  \begin{proof}[Proof of Theorem \ref{thm3}] We argue by induction on $n$. If $n=1$, the statement is obvious. Thus, we may and do assume that $n>1$. 
  Write $X= \prod_{i=1}^n X_i$ and let $\pi:Y\to X$ be a finite surjective ramified morphism.  It suffices to show that $X(k)\setminus \pi(Y(k))$ is dense in $X$. By Lemma \ref{lem:dom},  there is an integer $j\in \{1,\ldots, n\}$ such that the branch locus of $\pi:Y\to X$ dominates $\widetilde{X}_j$ (Definition \ref{defn}). Define $S:=\widetilde{X}_j$ and consider the natural morphism $p_j:X\to S$.  Note that, by the induction hypothesis, the smooth projective integral variety $S$ satisfies the very-weak-Hilbert property over $k$. Moreover,  for every $s$ in $S(k)$, the projective variety $X_s$ is naturally isomorphic to  $X_j$, and is therefore a smooth projective integral variety over $k$ satisfying the weak-Hilbert property over $k$. Thus, conditions $(1), (2), (3)$ of Theorem \ref{prop:thm} are satisfied. We conclude that $X(k)\setminus \pi(Y(k))$ is dense in $X$. 
  \end{proof}

  \begin{remark}
  Let $X$ and $Y$ be smooth projective connected varieties over a finitely generated field $k$ of characteristic zero. If $X$ and $Y$ are special   in the sense of Campana \cite{CampanaOr0}, then $X\times Y$ is special. Moreover, the conjectures of Campana and Corvaja-Zannier predict that $X$ is special if and only if there is a finite field extension $L/k$ such that $X_L$ has the weak-Hilbert property over $L$. In particular, Theorem \ref{thm3} is in accordance with the conjectures of Campana and Corvaja-Zannier as it verifies that a product of varieties with the weak-Hilbert property satisfies the very-weak-Hilbert property.
  \end{remark}
 
 We  now prove Theorem \ref{thm2}. Note that the proof is a straightforward application of  Theorem \ref{thm3} and Faltings's finiteness theorems for higher genus curves.

   \begin{proof}[Proof of Theorem \ref{thm2}] As in the statement of the theorem, we let $k$ be a finitely generated field of characteristic zero. Moreover, 
  let $E_1,\ldots, E_n$ be elliptic curves over  $k$ of positive rank over $k$.  Then, for every $i=1,\ldots, n$, the elliptic curve $E_i$ satisfies the weak-Hilbert property over $k$ by Faltings's theorem \cite{Faltings2, FaltingsComplements}.  (Indeed, it suffices to note that, if $E$ is an elliptic curve over $k$ and  $\pi:Y\to E$ is a ramified finite surjective morphism,  then the set $Y(k)$ is finite.)
Thus,   it    follows from Theorem \ref{thm3} that $E_1\times \ldots\times E_n$ satisfies the very-weak-Hilbert property over $k$.
  \end{proof}

\section{Kawamata's theorem}

To prove that the product of two elliptic curves satisfies the weak-Hilbert property, we will use Kawamata's theorem on finite covers of abelian varieties. Note that Kawamata's theorem is a generalization of Ueno's fibration theorem for closed subvarieties of abelian varieties.
\begin{theorem}[Kawamata]\label{thm:kawamata_fibn} Let $K$ be an algebraically closed field of characteristic zero, and let $A$ be an abelian variety over $K$. 
Let $X$ be a normal algebraic variety over $K$ and let $X\to A$ be a finite morphism. Then  there exist
\begin{enumerate}
\item an abelian subvariety $B$ of $A$;
\item finite \'etale Galois covers $X'\to X$  and $B'\to B$;
\item  a normal projective variety $Y$ of general type over $K$;
\item    a finite morphism $Y\to A/B$ with $A/B$ the quotient of $A$ by $B$ such that  $X'$ is a fiber bundle over $Y$ with fibers $B'$ and with translations by $B'$ as structure group
\end{enumerate} such that the following diagram
\[
\xymatrix{X' \ar[dd]_{B'-\textrm{fiber bundle}}
 \ar[rr]^{\textrm{finite étale}}   & &X \ar[rr]^{\textrm{finite}} & & A \ar[dd] \\
    & & & &  \\
  Y \ar[rrrr]_{\textrm{finite}} &  & & & A/B 
}
\]
commutes. 
\end{theorem}
\begin{proof}
See \cite[Theorem~23]{KawamataChar}.
\end{proof}

\begin{lemma}\label{lem:kaw2} Let $k$ be a finitely generated field of characteristic zero, let $A$ be an abelian surface over $k$, and let $Y\to A$ be a finite surjective ramified morphism with $Y$ integral normal. If the Kodaira dimension of $Y$ is not two, then $Y(k)$ is not dense.
\end{lemma}
\begin{proof} Note that the Kodaira dimension of $Y$ is non-negative, as $Y$ admits a finite surjective morphism to an abelian variety.   If the Kodaira dimension of $Y$ is zero, then $Y\to A$ is \'etale by Kawamata's theorem (Theorem \ref{thm:kawamata_fibn}). This contradicts our assumption that $Y\to A$ is ramified. Thus, we may and do assume that the Kodaira dimension of $Y$ equals one.
Then, 
by Kawamata's theorem (Theorem \ref{thm:kawamata_fibn}),  there is a  finite field extension $L/k$ and a finite \'etale cover $Y'\to Y_L$ of the surface $Y_L$ such that $Y'$ dominates a curve $C$  over $L$ of genus at least two. By Chevalley-Weil \cite[\S 8]{JLitt},  if $Y(k)$ is dense,  then there is a finite field extension $M/k$ such that $Y'(M)$ is dense.  As $Y'\to C$ is surjective, it follows that $C(M)$ is dense. However, this contradicts Faltings's theorem \cite{FaltingsComplements} that $C(M)$ is finite. We conclude that the set $Y(k)$ is not dense in $Y$. 
\end{proof}

\begin{lemma}\label{lem:rh}
Let $A$ and $B$ be elliptic curves over $k$, and let $\pi:Y\to A\times B$ be a finite surjective morphism with $Y$ of general type. Then the branch locus of $\pi$ dominates $A$ and $B$.
\end{lemma}
\begin{proof} Let $\psi:\widetilde{Y}\to Y$ be a resolution of singularities, and let $E= E_1\cup \ldots \cup E_n$ be the exceptional locus. Let $R$ be the ramification divsor of $\pi:Y\to A\times B$. Then, by Riemann-Hurwitz, we have that 
\[K_Y = \pi^\ast K_{A\times B} +R = R, \quad
K_{\widetilde{Y}} = \psi^\ast R + \sum a_i E_i.
\] As the canonical divisor $K_{\widetilde{Y}}$ is big on $\widetilde{Y}$ (as $\widetilde{Y}$ is of general type), we see that $\pi_\ast R$ is big on $A\times B$. Now, assume that the branch locus of $\pi$ does not dominate $A$. Then, the big divisor $\pi_\ast R$ is contained in $S\times B$ with $S$ a finite closed subset of $E_1$. However, as $S\times B$ is not big, this contradicts the bigness of $\pi_\ast R$. We conclude that the branch locus of $\pi$ dominates $A$ (hence also $B$ by symmetry).
\end{proof}

\begin{proof}[Proof of Theorem \ref{thm33}]   
Define $X:= E_1\times E_2$ and $S:=E_1$. Let $f:X\to S$ be the projection map.
For $i=1,\ldots,n $,
 let $Y_i$ be an  integral   normal variety over $k$   and let 
$\pi_i:Y_i\to X$ be a finite surjective ramified morphism.  It suffices to show that $X(k) \setminus \cup_{i=1}^n \pi_i(Y_i(k))$ is dense in $X$. To this end, let us first note that $X(k)$ is dense in $X$ (as $E_1(k)$ and $E_2(k)$ have positive rank). Now,  if $Y_i$ has Kodaira dimension $<2$, then $Y_i(k)$ is not dense (Lemma \ref{lem:kaw2}), so that we may discard such $Y_i$ from the collection of coverings $\pi_i:Y_i\to X$. That is, we may and do assume that, for $i=1,\ldots, n$, the variety $Y_i$ is of general type.  Moreover, if $Y_i\to T_i\to E_1$ is the Stein factorization of the composed morphism $Y_i\to E_1\times E_2\to E_1$ and $T_i\to E_1$ is ramified, then $Y_i(k)$ is not dense in $Y$, as $T_i(k)$ is finite by Faltings's finiteness theorem  \cite{Faltings2, FaltingsComplements}. Therefore, we may also discard such morphisms $\pi_i:Y_i\to X$ from the collection of coverings $\pi_i:Y_i\to X$.
Thus, for $i=1,\ldots, n$, the morphism $T_i\to S$ is finite unramified, hence \'etale.  Moreover, as $Y_i$ is of general type,  
by Lemma \ref{lem:rh}, for $i=1,\ldots, n$, the branch locus of $\pi_{i}$ dominates $S:=E_1$.  We now argue similarly as in the end of the proof of   Theorem \ref{prop:thm}.

Let $U\subset S$ be a dense open subset such that, for every $s$ in $U$, the scheme $Y_s$ is normal. 
For $s\in U(k)$, consider the finite surjective morphism $\pi_{i,s}:Y_{i,s}\to X_s$. Let $\{t_{i,1},\ldots, t_{i,r_i}\} = \psi_i^{-1}\{s\}$. Then $Y_{i,s} = Y_{t_{i,1}}\sqcup \ldots \sqcup Y_{t_{i,r_i}}$ with $Y_{t_{i,1}}, \ldots, Y_{t_{i,r_i}}$ integral normal varieties over $k$. Moreover, for every $i=1,\ldots, n$, every $s\in U(k)$, and every integer $1\leq j \leq  r_i$, by Lemma \ref{lemmatje}, the  induced finite surjective morphism $\pi_{i,s,j}:Y_{t_{i,j}}\to X_s$ is ramified (as the branch locus of $Y_i\to X$ dominates $S$, so that the branch locus of $Y_{i,T}\to X_T$ dominates $T$).   Therefore, 
since $X_s = E_2$ satisfies the weak-Hilbert property over $k$ (by assumption), it follows that 
\[
X_s(k)\setminus \cup_{i=1}^n\cup_{j=1}^{r_i} \pi_{i,s,j}(Y_{i,t_j}(k)) = X_s(k) \setminus \cup_{i=1}^n\pi_{i,s}(Y_{i,s}(k))
\] is dense in $X_s$. Note that, for every $s$ in $U(k)$, the set $X(k)\setminus \cup_{i=1}^n\pi_i(Y_i(k))$ contains  the set
\[
 X_s(k) \setminus \cup_{i=1}^n\pi_{i,s}(Y_{i,s}(k)).\]
Since $S(k) = E_1(k)$ is dense in $E_1$, we have that $U(k)$ is dense in $E_1$, so that    $$X(k)\setminus \cup_{i=1}^n\pi_i(Y_i(k))$$ is dense in $X$.
\end{proof}

 \bibliography{refsci}{}
\bibliographystyle{alpha}

\end{document}